\documentclass[12pt]{amsart}

\usepackage[margin=1.15in]{geometry}

\usepackage{amsmath,amscd,amssymb,amsfonts,latexsym}
\usepackage{wasysym}
\usepackage{mathrsfs}
\usepackage{mathtools,hhline}
\usepackage{color}
\usepackage{bm}
\usepackage[all, cmtip]{xy}
\usepackage{comment}

\usepackage{url,mathtools,amsmath}

\definecolor{hot}{RGB}{65,105,225}

\usepackage[pagebackref=true,colorlinks=true, linkcolor=hot ,  citecolor=hot, urlcolor=hot]{hyperref}%make all the references and links clickable

\theoremstyle{plain}
\newtheorem{theorem}{Theorem}[section]

\newtheorem{lemma}[theorem]{Lemma}

\theoremstyle{definition}
\newtheorem{definition}[theorem]{\sc Definition}
\newtheorem{example}[theorem]{\sc Example}

\numberwithin{equation}{section}

\newcommand\hot{\mathrm{h.o.t.}}

\newcommand\sW{\mathscr{W}}

\newcommand\ity{\infty}

\def\bC{\mathbb{C}}
\def\bP{\mathbb{P}}

\def\m{\setminus}
\def\s{\subset}

\renewcommand{\d}{{\mathrm d}}
\newcommand{\fin}{\hspace*{\fill}$\square$}

\DeclareMathOperator{\Sing}{Sing}                                    
\DeclareMathOperator{\Jac}{Jac}
\DeclareMathOperator{\mult}{mult}

\DeclareMathOperator{\reg}{reg} 
\DeclareMathOperator{\rank}{rank}

\title[]{Local linear Morsifications}

\author{Mihai Tib\u ar}
\address{M. Tib\u{a}r: Universit\' e de  Lille, CNRS, UMR 8524 -- Laboratoire Paul Painlev\'e, F-59000 Lille, France}  
\email {mihai-marius.tibar@univ-lille.fr}
\thanks{We acknowledge partial support from the Labex CEMPI (ANR-11-LABX-0007). }


\begin{document}

\date{\today}

\begin{abstract}  
The number of Morse points in a Morsification determines the topology of the 
Milnor fibre of a  holomorphic function germ $f$ with isolated singularity. If $f$ has an arbitrary singular locus, then this nice relation to the Milnor fibre disappears. We show that despite this loss, 
 the  numbers of stratified Morse singularities of a general linear Morsification
are effectively computable in terms of topological invariants of $f$. 
\end{abstract}

\maketitle

%%%%%%%%%%%%%%%

\section{Introduction}

One  may deform a holomorphic function germ $f: (\bC^{n+1}, 0) \to (\bC, 0)$ in a continuous one-parameter family $f_{\lambda}: (\bC^{n+1},0)  \to (\bC,0)$ with $f_{0}=f$, such that all the singularities of $f_{\lambda}$ close to the origin are Morse, for any small enough $\lambda \not= 0$.  Such a deformation is called \emph{Morsification of $f$} and a typical question one may ask is:
 \emph{how many Morse points  of $f_{\lambda}$ converge to the origin $0 \in \bC^{n+1}$ when $\lambda \to 0$?} 
 %Let us call it \emph{a Morse number of $f$}.

% Computing the number of singular points of a Morse function is a task of wide interest. 
%The case of holomorphic function germs $g$ with isolated singularity has been  extensively studied starting with %Milnor's book \cite{Mi}...  
In case $f$ has \emph{isolated singularity}, Brieskorn showed in \cite[Appendix]{Bri}  that  this number of Morse points is precisely the \emph{Milnor number} of $f$ at 0. %, denoted by $\mu(f)$.
This concidence is based on the constancy of  the topology of the general fibre of $f_{\lambda}$ inside a fixed ball. 
% in case of isolated singularities, Brieskorn principle of proof  identifies the top Betti number of the Milnor fibre of $f$ to the %number of Morse points in any Morsification of the function germ. 

In the case $f$ has nonisolated singularities,
 while Morsifications still exist, there is no more conservation of the fibre topology and it was unknown whether one can still have some topological control over this phenomenon.
%While Morsifying a nonisolated singularity destroys the nearby topology of the function, and  there 
%seems no hope to relate the topology of the Milnor fibration of $f$ to that of $f_{\lambda}$,  
%more difficult to express the number of Morse points of $f_{\lambda}$ in terms of topological invariants of $f$.
In particular: \emph{can one count the Morse points of a Morsification of $f$ in terms of topological invariants of $f$ only?}

\

This type of question occurs in real geometry for the distance function $d_{u}$  in \cite{DHOST}, and gives rise to the well-known by now \emph{Euclidean distance degree}, abbreviated ED-degree. 
In case the centre $u$ of the distance function $d_{u}$ cannot be chosen 
in a general position, the situation is similar to having an initial function with \emph{non-isolated singularities}, and a linear deformation of it.

  In the global setting of a complex polynomial $P$ on a complex affine variety $X\subset \bC^{N}$,  a topological interpretation  of the ED-degree in terms of Euler obstruction has been found by Maxim,  Rodriguez and Wang in \cite{MRW2018}.
For a linear deformation of a complex polynomial function, they have computed in \cite{MRW5} the number of Morse points  on the regular part $X_{\reg}$ of a singular affine space $X$, under the condition that no Morse point escapes to infinity, in terms of finitely many local multiplicities $n_{V}$ at strata $V\subset \Sing P$.  These multiplicities are not easy to grip; they
have been further investigated in \cite{MT1}. Formulas rely on relatively heavy computations of vanishing cycles based on the Euler obstruction.   A different way of computing the stratified Morse numbers is developed in \cite{MT3} based on polar curve techniques. 

\

We  address here the case of a holomorphic function germ $f$ with nonisolated singularity and remark, first of all, that the number of Morse points clearly depends on the type of deformation. Example: 
 $f(x,y) = x^{3}$. The family $f_{\lambda} = x^{3}-\lambda (x+by)$ for  $b\not= 0$, is a linear Morsification of $f$, and the number of Morse points of $f_{\lambda}$ is equal to zero. 
%On the other hand, the polar curve $\Gamma(x+by,f)$ is empty.  This confirms Theorem .....   
 If instead of a linear Morsification we consider a quadratic Morsification of $f$, for instance  $F_{\lambda} := f - \lambda (x^{2}+y^{2})$, then $F_{\lambda}$ has two Morse points, for $\lambda \not= 0$.

If we focus to linear Morsifications, i.e. deformations of the type $f_{\lambda} := f -\lambda \ell$ for  some general linear function $\ell$, then the number of Morse points acquires a precise meaning.
Moreover, our setting will be fully  general: let  $(X,0) \subset (\bC^{N},0)$ be a singular analytic set germ, and 
let $f:  (X,0) \to (\bC,0)$ be a holomorphic function germ. For a generic $\ell$, the linear deformation $f_{\lambda}$ is a Morsification,
%We address  the above question from a different viewpoint within the class of general linear Morsifications.
and the number of stratified Morse points of $f_{\lambda}$ which converge to the origin when $\lambda \to 0$  is stratwise constant (cf.  Section \ref{ss:local} for the definitions).
 We address  here the above question from a radically different viewpoint within the class of linear Morsifications. We
 give a method for computing the  \emph{Morse numbers} $m_{V}(f)$, i.e. the numbers of Morse points which abut to 0 when $\lambda \to 0$  on each stratum $V$ of the canonical stratification $\sW$ of $X$. Our  Theorem \ref{t:main1} shows the following formula in terms of polar multiplicities:
\begin{equation}\label{eq:main}
 m_{V}(f) = \mult_{0}\Bigl(\Gamma_{V}(\ell,f), \{f =0\}\Bigr) - \mult_{0}\Bigl(\Gamma_{V}(\ell,f), \{\ell =0\} \Bigr),
\end{equation}
where $\Gamma_{V}(\ell,f)$ is the generic polar curve of $f$ restricted to the stratum $V$ (cf Section \ref{ss:local} for details). Surprisingly, this formula looks exactly as if  $f$ were with isolated singularity.  There are two known proofs of \eqref{eq:main} in case $f$  has \emph{isolated singularity}: Massey's proof \cite{Ma} with rather involved vanishing cycles computations, and a much shorter one  in \cite{MT2}, by induction and using the comparison between two types of bouquet structure formulas for the Milnor fibre. Both proofs are based on the constancy of the Milnor fibre, already pointed out above, which cannot be exploited anymore when $f$ has nonisolated singularities. The same principle has been used recently by Zach \cite{Za} for computing the Morse numbers $m_{V}(f)$ through a cohomological method, still in case of an isolated singularity. The proof that we propose here is also different than 
the proofs developed \cite{MT3} for a far more general setting;  we hope that it will interest the reader by its effectivity side. 

A simple formula like \eqref{eq:main} in case of a function germ $f$ with higher dimensional singular locus was hardly expected because the fibre topology of $f$ is destroyed by the Morsification $f_{\lambda}$.  A change of paradigm was clearly needed for addressing this setting. 
 
Long time awaited since Brieskorn's result in 1970 for the case of isolated singularities, the solution to the aforementioned   question of computing the stratwise Morse numbers of $f$ in case of nonisolated singularities goes back to the fundamentals and offers a new perspective over the phenomenon of Morsification in the complex setting, together with a new method with far-reaching possibilities of applications due to its effectivity.

\medskip 

\noindent {\bf Acknowlegment.}  The author acknowledges support from the project ``Singularities and Applications'' - CF 132/31.07.2023 funded by the European Union - NextGenerationEU - through Romania's National Recovery and Resilience Plan, and support by the grant CNRS-INSMI-IEA-329.

%%%%%%%%%%%%%%%%%%%%%%%%%%%%%%%%
%%%%%%%%%%%%%%%%%%%%

\section{Linear Morsifications}\label{ss:local}
  Let $(X,0) \s (\bC^{N},0)$ be a singular irreducible space germ  of pure dimension $n+1\ge 2$, and let  $f:(X, 0)\to (\bC, 0)$ be a non-constant holomorphic function germ.

One may endow some small neighbourhood of $X$ with a Whitney stratification with finitely many strata  such that its regular part $X_{\reg}$ is a stratum. The roughest such stratification (with respect to inclusion of strata) exists,  it is called \emph{the canonical Whitney stratification of $X$}, and we will use it here under the notation $\sW$.

 Let $\Sing_{\sW}f :=\bigcup_{V\in \sW} \Sing f_{|V}$ denote the stratified singular locus of $f$ with respect to $\sW$. It is a closed set that we view as a set germ at the origin.

% One may refine $\sW$ to a Whitney stratification  $\sS$ such  that $\Sing_{\sW}f$ becomes a union 
%of strata of $\sS$, and we will occasionally write  $\Sing_{\sS}f$ instead $\Sing_{\sW}f$ when 
%strata of $\sS$ included in $\Sing_{\sW}f$ are involved.

Let us define what we mean by \emph{general linear Morsification} of $f$.
 
\begin{definition}[Stratified Morse function, after Goresky and MacPherson \cite{GM}]\label{d:morsefunction} \ \\
We say that a holomorphic function $h: X\cap B \to \bC$ defined on some small neighbourhood $B$ of 0 is a \emph{stratified Morse function} with respect to the stratification $\sW$ if $h$ has only stratified Morse singularities on the positive dimensional strata of $\sW$, and $h$ is general at $0$.   
\end{definition}

 For some linear function $\ell:\bC^{N}\to \bC$, we consider the 
 map $(\ell,f) \colon (X,0) \to (\bC\times \bC, (0,0))$ and its stratified singular locus $\Sing_{\sW}(\ell,f) := \bigcup_{V\in \sW} \Sing (\ell,f)_{|V}$, where 
 $$\Sing (\ell,f)_{|V} := \{x\in V \mid \rank \Jac(\ell_{|V},f_{|V}) <2 \}.$$
 Then $\Sing_{\sW}(\ell,f)$ is a closed set due to the Whitney regularity of the stratification $\sW$, and we will refer to it as a set germ at the origin. 

%%%%%%
 
\begin{definition}[Polar locus]\label{d:polarlocus}
One says that  
$$\Gamma(\ell,f) :=  \overline{ \Sing_{\sW} (\ell,f) \setminus \Sing_{\sW} f } \subset X$$ 
is the polar set of $f$ with respect to the function $\ell$, and we will refer to it as a set germ at the origin.
\end{definition}

  %%%%%%%%%%%%%%%%%%%%
  %%%%%%%%%%%%%%%%%%%%
% \section{Stratified Morse numbers}\label{s:local} \label{ss:local}

%a non-constant  polynomial function $f:X\to \bC$.

 The following fundamental result of Bertini-Sard type goes back to Hamm and L\^{e} \cite{HL} and  Kleiman \cite{Kl}, see also e.g. \cite{Ti-israel}:

\begin{lemma}[Local Polar Curve Lemma]  \label{l:polargen} \ \\
There is a Zariski open dense subset $\Omega'   \subset \check \bP^{N-1}$ such that the polar locus $\Gamma(\ell,f)$ is either a curve for all $\ell \in \Omega'$, or is empty for all $\ell \in \Omega'$.  

In the non-empty case,  there exists moreover a Zariski open subset $\Omega \subset \Omega'$ 
such  that $\Gamma (\ell,f)$ is reduced, and that the restriction $(\ell,f)_{| \Gamma (\ell,f)}$ is  one-to-one. 
\fin
\end{lemma}

%There is a similar local result for germs of holomorphic functions $f:(X, 0)\to (\bC, 0)$.

We will say that $\ell\in \Omega$  is a ``general linear function''.
If the polar curve germ $\Gamma (\ell,f)$ is not empty then it decomposes as: 
$$\Gamma (\ell,f) = \bigcup_{V\in \sW}\Gamma_{V} (\ell,f)$$
 where  $\Gamma_{V} (\ell,f)$ denotes the union of the irreducible curve components that are included in $V\cup \{0\}$.

 %%%%%%%%%%%%%%%%%%%%%
 %%%%%%%%%%
 %\section{Morse numbers of a function germ}\label{s:local}

\begin{definition}\label{d:genlinear}
Let $f:(X, 0)\to (\bC, 0)$ be a non-constant holomorphic function germ. We say that the family of  holomorphic function germs $f_{\lambda}:= f - \lambda \ell$  is a 
 \emph{general linear Morsification of $f = f_{0}$} if $\ell \in \Omega$ and $\lambda$ is close enough to $0\in \bC$.
\end{definition}
 
From now on we will consider general linear Morsifications  only.  By comparing the definitions of the polar locus $\Gamma_{V} (\ell,f)$ with that of the singular locus $\Sing ({f_{\lambda}}_{|V})$, one concludes that the Morse points of the restriction  ${f_{\lambda}}_{|V}$ belong to $\Gamma_{V} (\ell,f)$.
  
We then consider the following numbers:
 \begin{equation}\label{eq:morsenb}
 m_{V}(f) := \# \bigl\{ \mbox{Morse points of  the restriction } (f_{\lambda})_{|V} \mbox{ which converge to $0$ as } \lambda\to 0\bigr\}.
\end{equation}

From the definition it follows that the numbers $m_{V}(f)$ are independent of the generic $\ell\in \Omega$. We will call them
\emph{the stratified Morse numbers of $f$}.  They have been introduced  in \cite{MT2} in case of isolated singularities. What we prove here is that, surprisingly,  the  same  polar formula \cite[Theorem 3.1]{MT2} holds for the  stratified Morse numbers in the fully general setting of any singular locus of $f$.  However one needs a totally different principle of proof.
\begin{theorem}\label{t:main1}
Let  $f:(X, 0) \to (\bC, 0)$ be some singular non-constant holomorphic function germ,  let $\ell \in \Omega$ be a general linear function,  and let $V\in \sW$ be a positive dimensional stratum. Then:
\begin{equation}\label{eq:main1}
  m_{V}(f) = \mult_{0}\Bigl(\Gamma_{V}(\ell,f), \{f =0\}\Bigr) - \mult_{0}\Bigl(\Gamma_{V}(\ell,f), \{\ell =0\} \Bigr).
\end{equation}
\end{theorem}
We tacitly use the convention that the multiplicity $\mult_{0}\left(\Gamma_{V}(\ell,f), * \right)$ is zero  if  $\Gamma_{V}(\ell,f) = \emptyset$. In particular, if $\Gamma_{0,V} (\ell,f)$ is empty
then  $m_{V}(f) =0$.

%%%%%%%%%%%%%%%%%%
\subsection{Proof of Theorem \ref{t:main1}}

 Let $f_{\lambda}(x) = f(x) - \lambda \ell(x)$ be a general linear 
Morsification of $f$, for $\ell\in \Omega$. 

%%%%%%%%%%
 
\subsection{Sets of stratified Morse points}

Our method for computing the numbers $m_{V}(f)$ uses the following  convenient identification of the sets of stratified Morse points.
%%%%%%%
\begin{lemma}\label{t:polardiscrim-new}
 Let $\ell \in \Omega$, and let $V\in \sW$ be a positive dimensional stratum of $X$.\\
 The set of Morse singularities  of $f_{\lambda}$ on $V$  is:
    $$\Bigl\{ p\in V\cap \Gamma_{V}(\ell,f) \   \bigl\vert  \  \mult_{p}\bigl(\Gamma_{V}(\ell,f), \bigl\{f_{\lambda |V} = f_{\lambda}(p)\bigl\} \bigr) = 2 \Bigr\}.$$
\end{lemma}
%%%%%%%%%
\begin{proof} 
 The Morse points of the restriction of $f_{\lambda} = f- \lambda\ell$ to some stratum $V$ satisfy the equations of the polar curve $\Gamma_{V} (\ell,f)$, thus  a Morse point $p=p(\lambda)\in V$  has a trajectory inside $\Gamma_{V}(\ell,f)$ which abuts to the origin as $\lambda\to 0$. 
 
  We consider now the restrictions of all functions to the smooth stratum $V$.  The branches of the polar curve inside $V\cup \{0\}$, that we have denoted by $\Gamma_{V}(\ell,f)$, are non-singular outside $0$, in particular non-singular at $p(\lambda)$ for $\lambda\not= 0$ close enough to 0.  Since $\ell_{|V}$ is general with respect to $f_{\lambda |V}$ at $p$, we may apply the classical results \cite{Le}, \cite{Ti-israel} to the 
singular fibration defined by  the map germ  $(\ell,f_{\lambda})_{|V}$ at $p$.  
  More precisely, the restriction $f_{\lambda |V}$ has a Morse singularity at $p$, thus the function germ $f_{\lambda |V}$ at $p$ has Milnor number equal to 1. As the function germ $\ell_{|V}$ is non-singular  and transversal to the polar curve $\Gamma_{V}(\ell,f)$ at $p$, we have $\mult_{p}\bigl(\Gamma_{V}(\ell,f), \{\ell = \ell(p)\}\bigr) =1$. %By using  the equality  $\Gamma_{V}(\ell,f_{\lambda}) = \Gamma_{V}(\ell,f)$, 
The polar formula for the number of vanishing cycles\footnote{See also \cite{Ti-bouquet}, \cite{Ma}, or the ``comparison theorem''  in \cite{MT1}, \cite{MT2}.} applied at the Morse point $p$ of $f$ takes thus the form:
    \[  1 = \mult_{p}\Bigl(\Gamma_{V}(\ell,f_{\lambda}), \bigl\{ f_{\lambda |V} = f(p)\bigr\} \Bigr) - \mult_{p}\Bigl(\Gamma_{V}(\ell,f_{\lambda}), \bigl\{ \ell_{|V} =  \ell(p)\bigr\} \Bigr).
    \]
 It then follows  that $\mult_{p}\bigl(\Gamma_{V}(\ell,f_{\lambda}), f_{\lambda|V}^{-1}(f(p)\bigr) = 2$. We have the equality of set germs at $p$:  $\Gamma_{V, p}(\ell,f_{\lambda}) = \Gamma_{V, p}(\ell,f)$. Indeed,  the Morse point $p$ of $f_{\lambda}$ is not on $\Sing f$ (by the genericity of $\ell$) and the two sets are defined by the same equations in the neighbourhood of  $p$.  We therefore have shown: $\mult_{p}\bigl(\Gamma_{V}(\ell,f), f_{\lambda|V}^{-1}(f(p)\bigr) = 2$.
 
 Reciprocally, if $f_{\lambda |V}$ has no  (Morse) singularity at $p(\lambda)\in \Gamma_{V}(\ell,f)$ then, since its Milnor number at $p$ is 0,  we have, as above, the equality:
  \[  0 = \mult_{p}\Bigl(\Gamma_{V}(\ell,f_{\lambda}), \bigl\{ f_{\lambda |V} = f(p)\bigr\} \Bigr) - \mult_{p}\Bigl(\Gamma_{V}(\ell,f_{\lambda}), \bigl\{ \ell_{|V} =  \ell(p)\bigr\} \Bigr),
    \]
    which implies  $\mult_{p}\bigl(\Gamma_{V}(\ell,f), \bigl\{ f_{\lambda |V} = f(p)\bigr\}\bigr)=1$ after  replacing $\Gamma_{V}(\ell,f_{\lambda})$ by  $\Gamma_{V}(\ell,f)$ as  explained above. 
        The proof is complete.
\end{proof}

In view of the above proof,  the Morse points of $f_{\lambda}$ on $V$ are precisely \emph{the  tangency points between the polar curve $\Gamma_{V}(\ell,f)$ and the fibres of $f_{\lambda}$}.

%%%%%%%
\subsection{Reduction to $\bC^{2}$.}
We consider the map germ  $(\ell, f) : (X,0) \to (\bC^{2},0)$. The image by $(\ell, f)$ of the fibre $\{f_{\lambda} = \alpha\}$ is then the line $v -\lambda u =\alpha$, for any $\lambda \in \bC$, where $(u, v)$ denote the coordinates in the target $\bC^{2}$. For some positive dimensional stratum $V\in \sW$, we have the images:
 $$\Delta_{V} := (\ell,f)(\Gamma_{V}(\ell,f)) \ \mbox{ and  } \  \Delta := (\ell,f)(\Gamma(\ell,f)) .$$

   Let also $\gamma\subset \Gamma_{V}(\ell,f)$ denote some polar branch, and let $\delta_{\gamma} := (\ell,f)(\gamma) \subset \Delta_{V}$ be its image.
Since $\ell \in \Omega$,  the restriction  $(\ell, f)_{|} :\Gamma(\ell,f) \to \Delta$ is one-to-one (cf Lemma \ref{l:polargen}), which fact induces the equalities of intersection multiplicities: 
\begin{equation}\label{eq:arraymult}
 \begin{array}{lll}
 \mult _{0}\bigl(\Gamma_{V}(\ell,f),  f_{\lambda}^{-1}(0)\bigr) = \mult_{0}\bigl(\Delta_{V}, \{v -\lambda u=0\} \bigr),\\
 \mult _{0}\bigl(\Gamma_{V}(\ell,f),  \ell^{-1}(0)\bigr) = \mult_{0}\bigl(\Delta_{V}, \{u=0\} \bigr)
 \end{array} 
\end{equation}

Moreover, the map $(\ell, f)$ establishes the following one-to-one correspondence of finite sets in some small enough fixed pointed ball $B^{*}\subset X$ at the the origin, and for $\lambda$ close enough to 0:

$$ \left\{ \begin{array}{l} \mbox{ tangency point between } \\  \Gamma_{V}(\ell,f)
 \mbox{ and the fibres of }  f_{\lambda} \end{array} \right\} \stackrel{\sim}{\longmapsto}
 \left\{ \begin{array}{l} \mbox{ tangency point between } \Delta_{V}  \mbox{ and the fibres } \\
 \mbox{  of the linear function }  g(u,v) := v -\lambda u  \end{array} \right\} .
  $$

This implies that the tangency points of $\Gamma(\ell,f)$ to fibres of $f_{\lambda}$ which converge to the origin as $\lambda \to 0$ are 
in one-to-one correspondence with the tangency points of $\Delta$ to lines $v -\lambda u =\alpha$  which converge to 0 as $\lambda \to 0$. Moreover, this  bijective correspondence falls into one-to-one correspondences for each branch  $\gamma\in \Gamma(\ell,f)$ and its image $\delta_{\gamma}$. 

For any branch $\delta$ of $\Delta\subset \bC^{2}$, let us then define the number:
$$m_{\delta} := \# \left\{ \begin{array}{l} \mbox{points of  tangency of } \delta \mbox{ with  the fibres of }  \\ g(u,v) = v -\lambda u  \mbox{ which converge to } 0 \mbox{ as } \lambda \to 0 \end{array} \right\}.$$

As we have remarked before, in the above definition of  $m_{\delta}$ we understand by``point of tangency''
a point $q\in \Delta \m \{0\}$  such that   $\mult_{q}(\Delta, \{v -\lambda u= \alpha \})= 2$, where $\alpha = v(q) - \lambda u(q)$.

By using \eqref{eq:arraymult} and Lemma \ref{t:polardiscrim-new},  we  then deduce the following equality:

\begin{equation}\label{eq:sum}
 m_{V}(f) = \sum_{\delta \in \Delta_{V}}m_{\delta}.  
\end{equation}
  
%%%%%%%%%%%%
\subsection{Computation of $m_{\delta}$}

After \eqref{eq:sum}, computing $m_{V}(f)$ relies on that of  each $m_{\delta}$.
So let $\delta$ be a branch of $\Delta\subset \bC^{2}$, and let $\delta(s) : u= s^{i}, \ v= as^{j} + \hot$ be a local Puiseux parametrisation of it. We have here $j\ge i$ due to the fact that $\ell$ is a general function with respect to $f$, in the sense of the Polar Curve Lemma \ref{l:polargen}.  

A side comment is due here: the strict inequality $j> i$ holds in case $f\in \frak{m}_{X}^{2}$, where $\frak{m}_{X}$ denotes the  maximal ideal of germs of functions on $(X,0)$, and is a consequence of the tangency of $\Delta$ 
to the axis $\{v=0\}$. This fact has been proved in full generality in  \cite{Ti-carrousel}, and it allows to show the existence of a geometric monodromy of $f$ without fixed points, cf \cite{Ti-carrousel}. This had been previously proved by L\^{e} D.T. in \cite{Le} in the smooth case $X = \bC^{n}$.

\

%Recall that for some polynomial $g: \bC^{2}\to \bC$ $\grad g = (\partial g/\partial x_{1})$
The tangency condition occurring in the definition of $m_{\delta}$, between the parametrised arc $\delta(s)$  and the fibres of  the function $g(u,v) = v -\lambda u$, reads:
\[ \biggl \langle (- \overline{\lambda}, 1) , \bigl(\frac{\d u}{\d s},  \frac{\d v}{\d s}\bigr) \biggr \rangle=0,
\]
where $\overline{\lambda}$ denotes the complex conjugate of $\delta$, and $\langle \cdot , \cdot \rangle$ is the 
Hermitian inner product. This amounts to: 
\begin{equation}\label{eq:lambdaeq}
  - \lambda i s^{i-1} + a j s^{j-1} + \hot =0.
\end{equation}
  
By its definition, the value $m_{\delta}$ equals the number of non-zero solutions of  \eqref{eq:lambdaeq} which converge to 0 when $\lambda \to 0$. Thus,
after  dividing out in \eqref{eq:lambdaeq} by the factor $s^{i-1}$, we obtain the equation:
\[  - \lambda i + a j s^{j-i} + \hot =0
\]
which has precisely $j-i$ solutions that converge to 0 when $\lambda \to 0$.

On the other hand, by the definition of $\delta(s)$ we have :
\[ \mult_{0}\bigl(\delta, \{v=0\}\bigr) = j, \mbox{ and } \mult_{0}\bigl(\delta, \{u=0\}\bigr) = i.\]

We obtain the equality: 
$$m_{\delta} = \mult_{0}(\delta, \{v=0\})  - \mult_{0}(\delta, \{u=0\}).$$

By summing up these equalities over all $\delta \in \Delta_{V}$ as in \eqref{eq:sum}, we get:
\begin{equation}\label{eq:main2}
  m_{V}(f) = \mult_{0}\left(\Delta_{V}, \{v=0\}\right) - \mult_{0}\left(\Delta_{V},  \{u=0\} \right) .
\end{equation}
Using \eqref{eq:arraymult}, the equality \eqref{eq:main2}  lifts by the map $(\ell, f)$ to the equality \eqref{eq:main1}. This finishes our proof of Theorem \ref{t:main1}.

\begin{example}
 Let  $X= \bC^{2}$, and $f(x,y) = x^{k}y$, for some integer $k\ge 2$. The singular locus is $\Sing f = \{ x=0\}$. This is a $D_{\ity}$-singularity in Siersma's list  \cite{Si} of \emph{line singularities}. For $\ell = x+y$, the deformation $f_{\lambda}  := f - \lambda (x+y)$ is a general linear Morsification.
  The polar curve $\Gamma (\ell, f)$ is the line $\{ x= ky\}\subset \bC^{2}$, and by computing the intersection multiplicities of formula \eqref{eq:main1} we easily find that the Morse number of $f$ is $m_{\bC^{2}}(f) = k+1-1=k$, which may be confirmed by a direct computation of the Morse singularities of the function $f_{\lambda}$.
 
 Another example, still from Siersma's list of line singularities \cite{Si}, is  $J_{\ity}$:  $f(x,y) = x^{2} y^{2} + x^{3}$,  with the same singular locus $\Sing f = \{ x=0\}$.  Again $f_{\lambda} := f -\lambda (x+y)$ is a general linear Morsification. Here the polar curve $\Gamma (\ell, f)$ has equation $2xy - 2y^{2}-3x =0$.  Applying Theorem \ref{t:main1}, we get that the Morse number of $f$ is $m_{\bC^{2}}(f)= 6-1 =5$.
  \end{example}

%%%%%%%%%%%%%%%%%%%%%%%%%

\end{document}